\documentclass[10pt,a4paper]{article}
\usepackage[latin1]{inputenc}
\usepackage{anysize}
\usepackage{graphicx}
\marginsize{2.5cm}{2.5cm}{2cm}{1cm}

\usepackage{amssymb,amsmath,amscd,amsfonts,amsthm,mathrsfs,color}

\newcommand{\cc}[1]{\mathcal{#1}}

\newcommand{\ra}[1]{\overrightarrow{#1}}

\usepackage{thmtools, thm-restate}

\MakeRobust{\overrightarrow}

\numberwithin{equation}{section}
\newtheorem{theorem}{Theorem}
\newtheorem{lemma}[theorem]{Lemma}
\newtheorem{claim}[theorem]{Claim}

\usepackage{url}
\begin{document}
\title{Directed path-decompositions}
\author{
Joshua Erde%
\footnote{
University of Hamburg, 
Department of Mathematics, 
Bundesstra\"se 55 (Geomatikum),
20146 Hamburg, 
Germany.
\tt{joshua.erde@uni-hamburg.de}
}%
\footnote{The author was supported by the Alexander von Humboldt Foundation.}
}
\maketitle
\begin{abstract}
Many of the tools developed for the theory of tree-decompositions of graphs do not work for directed graphs. In this paper we show that some of the most basic tools do work in the case where the model digraph is a directed path. Using these tools we define a notion of a directed blockage in a digraph and prove a min-max theorem for directed path-width analogous to the result of Bienstock, Roberston, Seymour and Thomas for blockages in graphs. Furthermore, we show that every digraph with directed path width $\geq k$ contains each arboresence of order $\leq k$ 
as a butterfly minor. Finally we also show that every digraph admits a linked directed path-decomposition of minimum width, extending a result of Kim and Seymour on semi-complete digraphs.
\end{abstract}
\section{Introduction}
Given a tree $T$ and vertices $t_1,t_2 \in V(T)$ let us denote by $t_1Tt_2$ the unique path in $T$ between $t_1$ and $t_2$. Given a graph $G=(V,E)$ a \emph{tree-decomposition} is a pair $(T,\cc{V})$ consisting of a tree $T$, together with a collection of subsets of vertices $\cc{V} = \{ V_t \subseteq V(G) \, : \, t \in V(T)\}$, called \emph{bags}, such that: 
\begin{itemize}
\item $V(G) = \bigcup_{t \in T} V_t$;
\item For every edge $e \in E(G)$ there is a $t$ such that $e$ lies in $V_t$;
\item $V_{t_1} \cap V_{t_3} \subseteq V_{t_2}$ whenever $t_2 \in V(t_1 T t_3)$.
\end{itemize}

The \emph{width} of this tree-decomposition is the quantity $\max \{ |V_t| -1 \, : \, t \in V(T) \}$ and its \emph{adhesion} is $\max \{ |V_t \cap V_{t'}| \, : \, (t,t') \in E(T) \}$. Given a graph $G$ its \emph{tree-width} tw$(G)$ is the smallest $k$ such that $G$ has a tree-decomposition of width $k$. A \emph{haven of order $k$} in a graph $G$ is a function $\beta$ which maps each set $X \subseteq V(G)$ of fewer than $k$ vertices to some connected component of $G-X$ such that, for each such $X$ and $Y$, $\beta(X)$ and $\beta(Y)$ touch. That is,  either $\beta(X) \cap \beta(Y) \neq \emptyset$ or there is an edge between $\beta(X)$ and $\beta(Y)$. Seymour and Thomas \cite{ST93} showed that these two notions are dual to each other, in the following sense:

\begin{theorem}\label{t:treehav}[Seymour and Thomas]
A graph has tree-width $\geq k-1$ if and only if it has a haven of order $\geq k$.
\end{theorem}

When $T$ is a path we say that $(T,\cc{V})$ is a \emph{path-decomposition} and the \emph{path-width} pw$(G)$ of a graph $G$ is the smallest $k$ such that $G$ has a path-decomposition of width $k$. Whilst the path-width of a graph is clearly an upper bound for its tree-width, these parameters can be arbitrarily far apart. Bienstock, Robertson, Seymour and Thomas \cite{BRST91} showed that, as for tree-width, there is a structure like a haven, called a \emph{blockage of order $k$}, such that the path-width of a graph is equal to the order of the largest blockage.

Explicitly, for any subset $X$ of vertices in a graph $G$ let us write $\partial(X)$ for the set of $v \in X$ which have a neighbour in $V(G)-X$. Two subsets $X_1,X_2 \subset V(G)$ are \emph{complementary} if $X_1 \cup X_2 = V(G)$ and $\partial(X_1) \subseteq X_2$ (or, equivalently, $\partial(X_2) \subseteq X_1)$. A blockage of order $k$ is a set $\cc{B}$ such that:
\begin{itemize}
\item  each $X \in \cc{B}$ is a subset of $V(G)$ with $|\partial(X)| \leq k$;
\item if $X \in \cc{B}$, $Y \subseteq X$ and $|\partial(Y)| \leq k$, then $Y \in \cc{B}$;
\item if $X_1$ and $X_2$ are complementary and $|X_1 \cap X_2| \leq k$, then $\cc{B}$ contains exactly one of $X_1$ and $X_2$.
\end{itemize}

\begin{theorem}\label{t:pathblock}[Bienstock, Robertson, Seymour and Thomas]
A graph has path-width $\geq k$ if and only if it has a blockage of order $\geq k$.
\end{theorem}

The authors used the theorem to show the following result.
\begin{theorem}[Bienstock, Robertson, Seymour and Thomas]\label{t:min}
For every forest $F$, every graph with pathwidth $\geq |V(F)| -1$ has a minor isomorphic to $F$.
\end{theorem}
In particular, since large binary trees have large path-width, it is a simple corollary of this result that if $X$ is a graph, then the tree-width of graphs in
\[
\text{Forb}_{\preceq}(X) := \{G \, : \, X \text{ is not a minor of } G \}
\]
is bounded if and only if $X$ is a forest. A more direct proof of this fact, without reference to blockages, was later given by Diestel \cite{D95}.

There have been numerous suggestions as to the best way to extend the concept of tree-width to digraphs. For example, directed tree-width \cite{JRST01,JRST01a,R99}, D-width \cite{S05}, Kelly width \cite{HK08} or DAG-width \cite{O06,BDHKO12}. In some of these cases generalizations of Theorem \ref{t:treehav} have been studied, the hope being to find some structure in the graph whose existence is equivalent to having large `width'. However these results have either not given an exact equivalence \cite[Theorem 3.3]{JRST01}, or only apply to certain classes of digraphs \cite[Corollary 3]{S05}. 

In contrast, if we wish to decompose a digraph in a way that the model digraph is a directed path, it is perhaps clearer what a sensible notion of a `directed path-decomposition' should be. The following definition appears in Bar{\'a}t \cite{B06} and is attributed to Robertson, Seymour and Thomas. We note that it agrees with the definition of a DAG-decomposition, in the case where the DAG is a directed path. 
 
Given a digraph $D=(V,E)$ a \emph{directed path-decomposition} is a pair $(P,\cc{V})$ consisting of a path $P$, say with $V(P)= \{t_1,t_2, \ldots, t_n\}$, together with a collection of subsets of vertices $\cc{V} = \{ V_i \subseteq V(D) \, : \, i \in [n] \}$ such that: 
\begin{itemize}
\item $V(D) = \bigcup_{t \in V(P)} V_t$;
\item if $i<j<k$, then $V_i \cap V_k \subseteq V_j$;
\item for every edge $e=(x,y)\in E(P)$ there exists $i \leq j$ such $x \in V_i$ and $y \in V_j$.
\end{itemize}

The \emph{width} of a directed path-decomposition is $\max \{ |V_i| -1 \, : \, i \in [n] \}$ and its \emph{adhesion} is $\max \{ |V_i \cap V_{i+1}| \, : \, i \in [n-1] \}$. Given a digraph $D$ its \emph{directed path-width} dpw$(D)$ is the smallest $k$ such that $D$ has a directed path-decomposition of width $k$.

Motivated by Theorem \ref{t:pathblock}, Bar{\'a}t \cite{B06} defined a notion of a blockage in a digraph and showed that if the directed path width is at most $k-1$ then there is no `blockage of order $k$'. However he suggested that it was unlikely that the existence of such a `blockage of order $k$' was equivalent to having directed path-width at least $k$, at least for this particular notion of a blockage.

One of the problems with working with directed graphs is that many of the tools developed for the theory of tree-decompositions of graphs do not work for digraphs. However, we noticed that in the case of directed path-decompositions some of the most fundamental tools do work almost exactly as in the undirected case. More precisely, Bellenbaum and Diestel \cite{BD02} explicitly extracted a lemma from the work of Thomas \cite{T90}, and used it to give short proofs of two theorems: The first, Theorem \ref{t:treehav}, and the second a theorem of Thomas on the existence of linked tree-decompositions. In \cite{KS15} Kim and Seymour prove a similar theorem on the existence of linked directed path-decompositions for semi-complete digraphs. Their proof uses a tool analogous to the key lemma of Bellenbaum and Diestel.

This suggested that perhaps a proof of Theorem \ref{t:pathblock} could be adapted, to give an analogue for directed path-width. Indeed, in this paper we show that this is the case. Influenced by ideas from \cite{DO14a,DO14b} and \cite{AMNT09} we will define a notion of \emph{diblockage}\footnote{A precise definition will be given in Section \ref{s:path}.} in terms of orientations of the set of directed separations of a digraph. We will use generalizations of ideas and tools developed by Diestel and Oum \cite{DO14a,DO14b} to prove the following.

\begin{theorem}\label{t:main}
A directed graph has directed path-width $\geq k-1$ if and only if it has a diblockage of order $\geq k$.
\end{theorem}

We also use these ideas to give a result in the spirit of Theorem \ref{t:min}. An \emph{arborescence} is a directed graph in which there is a unique vertex $u$, called the \emph{root}, such that for every other vertex $v$ there is exactly one directed path from $u$ to $v$. A forest of arborescences is a directed graph in which each component is an arborescence. There is a generalisation of the notion of a minor to directed graphs called a \emph{butterfly minor}, whose precise definition we will defer until Section \ref{s:arb}.

\begin{restatable}{theorem}{arb}\label{t:arborescence}
For every forest of arborescences $F$, every digraph with directed path-width $\geq |V(F)| -1$ has a butterfly minor isomorphic to $F$.
\end{restatable}

Finally, we extend the result of Kim and Seymour on linked directed path-decompositions to arbitrary digraphs.

\begin{theorem}\label{t:mainlinked}
Every directed graph $D$ has a linked directed path-decomposition of width dpw$(D)$.
\end{theorem}

The paper is structured as follows. In Section \ref{s:path} we define directed path-decompositions and introduce our main tool of shifting. In section \ref{s:blo} we define our notion of a directed blockage and prove Theorem \ref{t:main}. In Section \ref{s:arb} we prove Theorem \ref{t:arborescence}. Finally, in Section \ref{s:linked} we discuss linked directed path-decompositions and prove Theorem \ref{t:mainlinked}.

\section{Directed path-decompositions}\label{s:path}
Following the ideas of Diestel and Oum \cite{DO14a} it will be more convenient for us to rephrase the definition of a directed path-decomposition in terms of directed separations. Given a digraph $D = (V,E)$, a pair $(A,B)$ of subsets of $V$ is a \emph{directed separation} if $A \cup B = V$ and there is no edge $(x,y) \in E$ with $x \in B\setminus A$ and $y \in A \setminus B$. Equivalently, every directed path which starts in $B$ and ends in $A$ must meet $A \cap B$.

For brevity, since in this paper we usually be considering directed graphs, when the context is clear we will refer to directed separations simply as separations. The \emph{order} of a separation $(A,B)$, which we will denote by $|A,B|$, is $|A\cap B|$ and we will write $\ra{S}_k$ for the set of separations of order $<k$ and define $\ra{S} := \bigcup_k \ra{S}_k$.

We define a partial order on $\ra{S}$ by 
\[
(A,B) \leq (C,D) \text{    if and only if    } A \subseteq C \text{ and } B \supseteq D.
\]
We will also define two operations $\wedge$ and $\vee$ such that, for $(A,B),(C,D) \in \ra{S}$
\[
(A,B) \wedge (C,D) = (A \cap C, B \cup D) \text{    and    } (A,B) \vee (C,D) = (A \cup B, C \cap D) .
\]
\begin{lemma}
If $(A,B)$ and $(C,D)$ are separations then so are $(A,B) \wedge (C,D)$ and $(A,B) \vee (C,D)$.
\end{lemma}
\begin{proof}
Let us prove the claim for $(A,B) \wedge (C,D)$, the proof for $(A,B) \vee (C,D)$ is similar. Firstly, since $A\cup B = V$ and $C \cup D = V$ it follows that 
\[ (A \cap C) \cup (B \cup D)= (A \cup B \cup D) \cap (C \cup B \cup D) = V \cap V = V.
\]

Secondly, let $(x,y) \in E$ with $y \in (A \cap C) \setminus (B \cup D)$. Then, $y \in A \setminus B$ and $y \in C \setminus D$ and so, since $(A,B)$ and $(C,D)$ are separations, it follows that $x \not\in B \setminus A$ and $x \not\in D \setminus C$. Hence, $x \not\in (B \cup D) \setminus (A \cap C)$ and so $(A \cap C, B \cup D)$ is a separation.
\end{proof}
Given some subset $\ra{S}' \subseteq \ra{S}$, we define an \emph{$\ra{S}'$-path} to be a pair $(P,\alpha)$ where $P$ is a path with vertex set $V(P) = \{ t_1,t_2,\ldots t_n\}$ and $\alpha: E(T) \rightarrow \ra{S}'$ is such that if $1 \leq i < j \leq n-1$ then $\alpha(t_i,t_i+1) \leq \alpha(t_j,t_{j+1})$. Note that, if we consider $P$ as a directed path, with edges $(t_i,t_{i+1})$ for each $i$ then $\alpha$ preserves the natural order on the edges of $P$. In this way we can view this notion as a generalisation of the $S$-trees of Diestel and Oum \cite{DO14a}.

We claim that, if $(P,\cc{V})$ is a directed path-decomposition then for every $1 \leq i \leq n-1$ the following is a separation
\[
(\bigcup_{j \leq i} V_j, \bigcup_{j \geq i+1} V_j).
\]
Indeed, by definition $\bigcup_i V_i = V(D)$ and if $(x,y)$ were an edge from  $\bigcup_{j \geq i+1} V_j \setminus \bigcup_{j \leq i} V_j$ to $\bigcup_{j \leq i} V_j \setminus \bigcup_{j \geq i+1} V_j$ then clearly there could be no $i<j$ with $x \in V_i$ and $y \in V_j$. In this way $(P,\cc{V})$ gives an $\ra{S}$-path by letting 
\[
\alpha(t_i,t_{i+1}) =  (\bigcup_{j \leq i} V_j, \bigcup_{j \geq i+1} V_j).
\]

Conversely, if $(P,\alpha)$ is an $\ra{S}$ path let us write $(A_i,B_i) : =  \alpha(t_i,t_{i+1})$, and let $B_0 = V= A_{n+1}$. Note that, if $i<j$ then $(A_i,B_i) \leq (A_j,B_j)$ and so $A_i \subseteq A_j$ and $B_i \supseteq B_j$.

\begin{lemma}
Let $(P,\alpha)$ be an $\ra{S}$ path and $A_i,B_i$ be as above. For $1 \leq i \leq n$ let $V_i = A_i \cap B_{i-1}$. Then $(P,\cc{V})$ is a directed path-decomposition.
\end{lemma}
\begin{proof}
Firstly, we claim that for $1 \leq j \leq n-1$, $\bigcup_{i=1}^j V_i = A_j$. Indeed, $V_1 = A_1 \cap B_0 = A_1$ and, if the claim holds for $j-1$ then
\[
\bigcup_{i=1}^j V_i  = A_{i-1} \cup V_i = A_{i-1} \cup (A_i \cap B_{i-1}) = A_i,
\]
since $A_{i-1} \cup B_{i-1} = V$ and $A_{i-1} \subseteq A_i$. As we will use it later, we note that a similar argument shows that $\bigcup_{i=j}^n V_i = B_{j-1}$. Hence,
\[
\bigcup_{i=1}^n V_i = A_{n-1} \cup V_n =  A_{n-1} \cup (A_n \cap B_{n-1}) =  A_{n-1} \cup B_{n-1} = V.  
\]

Next, suppose that $i < j <k$. We can write 
\[
V_i \cap V_k = (A_i \cap B_{i-1}) \cap (A_k \cap B_{k-1}).
\]
Since $(A_i,B_i) \leq (A_j,B_j)$, it follows that $A_i \subseteq A_j$ and since $(A_{j-1},B_{j-1}) \leq (A_{k-1},B_{k-1})$ and so $B_{k-1} \subseteq B_{j-1}$. Hence
\[
V_i \cap V_k = (A_i \cap B_{i-1}) \cap (A_k \cap B_{k-1}) \subseteq A_j \cap B_{j-1} = V_j.
\]

Finally, suppose for a contradiction there is some edge $(x,y)$ such that there is no $i \leq j$ with $x \in V_i$ and $y \in V_j$. Since $\bigcup V_i =  V$ there must be $i > j$ with $x \in V_i$ and $y \in V_j$. Pick such a pair with $i$ minimal. It follows that $x \in V_i \setminus \bigcup_{k=1}^{i-1} V_k = (A_i \cap B_{i-1}) \setminus A_{i-1} \subset B_{i-1} \setminus A_{i-1}$ by the previous claim. Also, $y \in V_j \setminus \bigcup_{k\geq i} V_k  =  (A_j \cap B_{j-1} ) \setminus B_{i-1} \subseteq A_{i-1} \setminus B_{i-1}$ since $A_j \subseteq A_{i-1}$. However, this contradicts the fact that $(A_{i-1},B_{i-1})$ is a separation.
\end{proof}

In this way the two notions are equivalent. We say that the \emph{width} of an $\ra{S}$ path $(P,\alpha)$ is the width of the path-decomposition $(P,\cc{V})$ given by $V_i = A_i \cap B_{i-1}$. The following observation will be useful.

\begin{lemma}\label{l:prop}
If $(P,\alpha)$ is an $\ra{S}$-path of width $< k-1$ with $\alpha(t_i,t_{i+1}) = (A_i,B_i)$ and $A_0=V=B_n$, then 
\begin{itemize}
\item $(P,\alpha)$ is an $\ra{S}_k$-path;
\item $(A_i,B_{i-1}) \in \ra{S}_k$ for each $1 \leq i \leq n$.
\end{itemize}
\end{lemma}
\begin{proof}
For the first we note that, since $(A_{i-1},B_{i-1}) \leq (A_i,B_i)$, it follows that $B_{i-1} \supseteq B_i$ and hence
\[
|A_i,B_i| = |A_i \cap B_i| \leq |A_i \cap B_{i-1}| <  k.
\]

For the second, note that, since $A_{i-1} \subseteq A_i$, $B_{i-1} \setminus A_i \subset B_{i-1} \setminus A_{i-1}$, and since $A_{i-1} \cup B_{i-1} = V$, $A_i \setminus B_{i-1} \subset A_{i-1} \setminus B_{i-1}$. Hence, there is no edge from $B_{i-1} \setminus A_i$ to $A_i \setminus B_{i-1}$ and so $(A_i,B_{i-1}) \in \ra{S}$. Finally, since $|A_i \cap B_{i-1}| < k$, $(A_i,B_{i-1}) \in \ra{S}_k$.

\end{proof}
\subsection{Shifting an $\ra{S}_k$-path}
One of the benefits of thinking of a directed path-decomposition in terms of the separations it induces, rather than the bags, is that it allows one to easily describe some of the operations that one normally performs on tree-decompositions.

Given an $\ra{S}$-path $(P,\alpha)$ with $V(P) = \{t_1, \ldots, t_n\}$, let us write $\alpha(t_i,t_{i+1}) = (A_i,B_i)$. We call $(A_1,B_1)$ the \emph{initial leaf separation} of $(P,\alpha)$ and $(A_{n-1},B_{n-1})$ the \emph{terminal leaf separation}. It will be useful to have an operation which transforms an $\ra{S}$-path into one with a given initial/terminal leaf separation.

Let $(P,\alpha)$ be as above and let $(A_i,B_i) \leq (X,Y) \in \ra{S}$. The \emph{up-shift of $(P,\alpha)$ onto $(X,Y)$ with respect to $(A_i,B_i)$} is the $\ra{S}$-path $(P',\alpha')$ where $V(P')=\{t'_i, t'_{i+1}, \ldots, t'_n\}$ and $\alpha'$ is given by
\[
\alpha'(t'_j,t'_{j+1}) := (A_j,B_j) \vee (X,Y) = (A_j \cup X,B_j \cap Y).
\]
It is simple to check that, if $i \leq j < k \leq n-1$ then $\alpha'(t'_j,t'_j+1) \leq \alpha'(t'_k,t'_{k+1})$, and so $(P',\alpha')$ is an $\ra{S}$-path. We note that the initial leaf separation of $(P',\alpha')$ is $\alpha'(t'_i,t'_{i+1}) = (A_i,B_i) \vee (X,Y) = (X,Y)$

Similarly if $(X,Y) \leq (A_i,B_i)$ the \emph{down-shift of $(P,\alpha)$ onto $(X,Y)$ with respect to $(A_i,B_i)$} is the $\ra{S}$-path $(P',\alpha')$ where $V(P') = \{t'_1,t'_2, \ldots, t'_{i+1}\}$ and  $\alpha'$ is given by
\[
\alpha'(t'_j,t'_{j+1}) := (A_j,B_j) \wedge (X,Y) = (A_j \cap X,B_j \cup Y).
\]
We note that the terminal leaf separation of $(P',\alpha')$ is $\alpha'(t'_i,t'_{i+1}) = (A_i,B_i) \wedge (X,Y) = (X,Y)$.

If $(P,\alpha)$ is an $\ra{S}_k$-path and $(P',\alpha')$ is an up/down-shift of $(P,\alpha)$ then, whilst $(P',\alpha')$ is an $\ra{S}$-path, it is not always the case that it will also be an $\ra{S}_k$-path, since the order of some of the separations could increase. We note that if $(A,B)$ and $(C,D)$ are separations then $\wedge$ and $\vee$ satisfy the following equality
\begin{equation}\label{e:sub}
|A,B| + |C,D| = |A \cup C,B \cap D| + |A\cap C, B \cup D|.
\end{equation}

Given a pair of separations $(A,B) \leq (C,D) \in \ra{S}$ let us denote by
\[
\lambda\big( (A,B), (C,D) \big) := \min \{ |X,Y| \, : \, (X,Y) \in \ra{S} \text{ and } (A,B) \leq (X,Y) \leq (C,D) \}.
\]
We say that $(X,Y)$ is \emph{up-linked} to $(A,B)$ if $(A,B) \leq (X,Y)$ and $|X,Y| = \lambda\big((A,B),(X,Y)\big)$. Similarly $(X,Y)$ is \emph{down-linked} to $(A,B)$ if $(X,Y)\leq (A,B)$ and $|X,Y| = \lambda\big((X,Y),(A,B)\big)$.

\begin{lemma}\label{l:shift1}
Let $(P,\alpha)$ be an $\ra{S}_k$-path with $V(P) = \{ t_1,t_2, \ldots, t_n\}$, with $\alpha(t_j,t_{j+1}) = (A_j,B_j)$ for each $j$. If $(X,Y)$ is up-linked to $(A_i,B_i)$ then the up-shift of $(P,\alpha)$ onto $(X,Y)$ with respect to $(A_i,B_i)$ is an $\ra{S}_k$-path and if $(X,Y)$ is down-linked to $(A_i,B_i)$ then the down-shift of $(P,\alpha)$ onto $(X,Y)$ with respect to $(A_i,B_i)$ is an $\ra{S}_k$-path.
\end{lemma}
\begin{proof}
By the discussion above it is clear that both are $\ra{S}$-paths, so it remains to show that the set of separations in the paths lie in $\ra{S}_k$. We will show the first claim, the proof of the second follows along similar lines. Let $(P',\alpha')$ be the up-shift of $(P,\alpha)$ onto $(X,Y)$ with respect to $(A_i,B_i)$, with $P' = \{t'_i, t'_{i+1}, \ldots, t'_n\}$.

Let $i \leq j \leq n-1$. We wish to show that $\alpha'(t'_j,t'_{j+1}) = (A_j \cup X,B_j \cap Y)$ is in $\ra{S}_k$. We note that since $(A_i,B_i) \leq (A_j,B_j)$ and $(A_i,B_i) \leq (X,Y)$, it follows that
\[
(A_i,B_i) \leq (A_j \cap X,B_j \cup Y) \leq (X,Y)
\]
and so, since $(X,Y)$ is up-linked to $(A_i,B_i)$, 
\[
|A_j \cap X,B_j \cup Y| \geq |X,Y|.
\]
Therefore, by (\ref{e:sub}), 
\[
|A_j \cup X,B_j \cap Y| \leq |A_j,B_j| <k.
\]
Hence $\alpha'(t'_j,t'_{j+1})=(A_j \cup X,B_j \cap Y) \in \ra{S}_k$.
\end{proof}

Recall that the width of an $\ra{S}$-path $(P,\alpha)$ is $\max_{1 \leq i \leq n} |A_i \cap B_{i-1}| - 1$. We would like to claim that, apart from the bag at the initial leaf in an up-shift, or the bag at the terminal leaf in a down-shift, shifting does not increase the size of the bags.

Again, this will not be true for general shifts however, if the assumptions of Lemma \ref{l:shift1} hold, then it will hold. The proof of the fact follows the proof of \cite[Lemma 6.1]{DO14b}, which itself plays the role of the key lemma of Thomas from \cite{BD02}. The fact that this lemma remains true for directed path decompositions is what allows us to prove our Theorems \ref{t:main} and \ref{t:mainlinked}.

\begin{lemma}\label{l:bags}
Let $(P,\alpha)$ be an $\ra{S}_k$-path with $V(P) = \{ t_1,t_2, \ldots, t_n\}$, let $\alpha(t_j,t_{j+1}) = (A_j,B_j)$ for each $j$, and let $B_0 = V = A_n$. Let $\omega_j = |A_j \cap B_{j-1}|$ be the size of the bags. Suppose that $(X,Y)$ is up-linked to $(A_i,B_i)$ and $(P',\alpha')$ is the up-shift of $(P,\alpha)$ onto $(X,Y)$ with respect to $(A_i,B_i)$, with $\alpha'(t'_j,t'_{j+1}) = (A'_j,B'_j)$ for each $j$, and $B'_i = V = A'_n$. If $\omega'_j = |A'_j \cap B'_{j-1}|$, then for each $i+1 \leq j \leq n$ $\omega'_j \leq  \omega_j$.

Similarly, suppose that $(X,Y)$ is down-linked to $(A_i,B_i)$ and $(P',\alpha')$ is the up-shift of $(P,\alpha)$ onto $(X,Y)$ with respect to $(A_i,B_i)$, with $\alpha'(t'_j,t'_{j+1}) = (A'_j,B'_j)$ for each $j$, and $B'_0 = V = A'_{i+1}$. If $\omega'_j = |A'_j \cap B'_{j-1}|$, then for each $1 \leq j \leq i$ $\omega'_j \leq  \omega_j$.
\end{lemma}
\begin{proof}
Again, we will just prove the first statement, as the proof of the second is analogous. Recall that $(A'_j,B'_j) = (A_j \cup X,B_j \cap Y)$ for each $i \leq j \leq n$. Hence, for $i+1 \leq j \leq n$
\[
\omega'_j = | (A_j \cup X) \cap (B_{j-1} \cap Y)| =  | A_j \cup X, B_{j-1} \cap Y|.
\]
Here we have used the fact that $(A_j,B_{j-1})$ is a separation, by Lemma \ref{l:prop}, to deduce that $(A_j \cup X, B_{j-1} \cap Y)$ is also a separation.

Now, since $(A_i,B_i) \leq (A_{j-1},B_{j-1}) \leq (A_j,B_j)$, it follows that $(A_i,B_i) \leq (A_j,B_{j-1})$. Therefore, since also $(A_i,B_i) \leq (X,Y)$,
\[
(A_i,B_i) \leq (A_j \cap X, B_{j-1} \cup Y) \leq (X,Y).
\]
Hence, since $(X,Y)$ is linked to $(A_i,B_i)$, 
\[
|A_j \cap X, B_{j-1} \cup Y| \leq |X,Y|.
\]
Therefore, by (\ref{e:sub}), 
\[
\omega'_j = | A_j \cup X, B_{j-1} \cap Y| \leq |A_j,B_{j-1}| = \omega_j.
\]
\end{proof}
\section{$\omega$-Diblockages}\label{s:blo}
In \cite{DO14b} the structures which are dual to the existence of $S$-trees are defined as orientations of the set of separations in a graph. That is a subset of the separations which, for each separation given by an unordered pair $\{A,B\}$, contains exactly one of $(A,B)$ or $(B,A)$. Heuristically, one can think of these orientations as choosing, for each separation $\{A,B\}$ one of the two sides $A$ or $B$ to designate as `large'. This idea generalises in some way the concept of tangles introduce by Robertson and Seymour \cite{RS91}.

In our case, since the directed separations we consider already have a defined `direction', we will define an orientation of the set of directed separations to be just a bipartition of the set of directed separations. However we will still think of a bipartition $\ra{S}_k = \cc{O}^+ \dot\cup \cc{O}^-$ as designating for each directed separation $(A,B) \in \ra{S}_k$ one side as being `large`, the side $B$ when $(A,B) \in \cc{O}^+$ and the side $A$ when $(A,B) \in \cc{O}^-$. Our notion of a directed blockage will then be defined as some way to make these choices for each $(A,B) \in \ra{S}_k$ in a consistent manner.

Let us make the preceding discussion more explicit. We define a \emph{partial orientation} of $\ra{S}_k$ to be a disjoint pair of subsets $\cc{O}=(\cc{O}^+,\cc{O}^-)$ such that $\cc{O}^+, \cc{O}^- \subseteq \ra{S}_k$. A partial orientation is an \emph{orientation} if $\cc{O}^+ \dot\cup \cc{O}^- = \ra{S}_k$. Given a partial orientation $\cc{P}$ of $\ra{S}_k$ let us write 
\[
\ra{S}_{\cc{P}} = \ra{S}_k \setminus (\cc{P}^+ \cup \cc{P}^-).
\]
We say a partial orientation $\cc{P} = (\cc{P}^+,\cc{P}^-)$ is \emph{consistent} if
\begin{itemize}
\item if $(A,B) \in \cc{P}^+$, $(A,B) \geq (C,D) \in \ra{S}_k$ then $(C,D) \in \cc{P}^+$;
\item if $(A,B) \in \cc{P}^-$, $(A,B) \leq (C,D) \in \ra{S}_k$ then $(C,D) \in \cc{P}^-$.
\end{itemize}
In the language of the preceeding discussion this formalises the intuitive idea that if $B$ is the large side of $(A,B)$ and $B \subseteq D$ then $D$ should be the large side of $(C,D)$ and similarly if $A$ is the large side of $(A,B)$ and $A \supseteq C$ then $C$ should be the large side of $(C,D)$.

An orientation $\cc{O}$ \emph{extends} a partial orientation $\cc{P}$ if $\cc{P}^+ \subseteq \cc{O}^+$ and $\cc{P}^- \subseteq \cc{O}^-$. Given $\omega \geq k$, let us define $\cc{P}_\omega = (\cc{P}_\omega^+,\cc{P}_\omega^-)$ by
\[
\cc{P}_\omega^+ = \{ (A,B) \in \ra{S}_k \, : \, |A| < \omega \} \text{   and    } \cc{P}_\omega^- =  \{ (A,B) \in \ra{S}_k \, : \, |B| < \omega \}.
\]
An \emph{$\omega$-diblockage} (of $\ra{S}_k$) is an orientation $\cc{O} = (\cc{O}^+ \cup \cc{O}^-)$ of $\ra{S}_k$ such that:
\begin{itemize}
\item $\cc{O}$ extends $\cc{P}_\omega$;
\item $\cc{O}$ is consistent;
\item if $(A,B) \in \cc{O}^+$ and $(A,B) \leq (C,D) \in \cc{O}^-$ then $|B \cap C| \geq \omega$.
\end{itemize}

We will show, for $\omega \geq k$, a duality between the existence of an $\ra{S}_k$-path of width $< \omega -1$ and that of an $\omega$-diblockage of $\ra{S}_k$. In the language of tree-decompositions, an $\ra{S}_k$-path of width $< \omega -1$ is a directed path-decomposition of width $< \omega -1$ in which all the \emph{adhesion sets} have size $<k$, where the adhesion sets in a tree-decomposition $(T,\cc{V})$ are the sets $\{V_t \cap V_{t'}\, : \, (t,t') \in E(T)\}$. Tree-decompositions of undirected graphs with adhesion sets of bounded size have been considered by Diestel and Oum \cite{DO14b} and Geelen and Joeris \cite{GJ16}. 

We require one more definition for the proof. Given a partial orientation $\cc{P}$ of $\ra{S}_k$ and $\omega \geq k$, we say that an $\ra{S}_k$-path $(P,\alpha)$ with $V(P)=\{t_1,\ldots,t_n\}$ and $\alpha(t_j,t_{j+1}) = (A_j,B_j)$ for each $j$ is \emph{$(\omega,\cc{P})$-admissable} if
\begin{itemize}
\item For each $ 2 \leq i \leq n-1$, $\omega_i := |A_i \cap B_{i-1}| < \omega$;
\item $(A_1,B_1) \in \cc{P}^+ \cup \cc{P}_\omega^+$; 
\item $(A_{n-1},B_{n-1}) \in \cc{P}^- \cup \cc{P}_\omega^-$;
\end{itemize}

When $\omega = k$ we call an $\omega$-diblockage of $\ra{S}_k$ a \emph{diblockage of order $k$}. In this way, when $\omega=k$ the following theorem implies Theorem \ref{t:main}.

\begin{theorem}\label{t:diblock}
Let $\omega \geq k \in \mathbb{N}$ and let $D=(V,E)$ be a directed graph with $|V| \geq k$. Then exactly one of the following holds:
\begin{itemize}
\item  $D$ has an $\ra{S}_k$-path of width $<\omega-1$;
\item  There is an $\omega$-diblockage of $\ra{S}_k$.
\end{itemize}
\end{theorem}
\begin{proof}
We will instead prove a stronger statement. We claim that for every consistent partial orientation $\cc{P}$ of $\ra{S}_k$ either there exists an $(\omega,\cc{P})$-admissable $\ra{S}_k$-path, or there is an $\omega$-diblockage of $\ra{S}_k$ extending $\cc{P}$. Note that, $\cc{P}_\omega$ is consistent, and an $(\omega,\cc{P}_\omega)$-admissable $\ra{S}_k$-path is an $\ra{S}_k$-path of width $<\omega-1$.

Let us first show that both cannot happen. Suppose for contradiction that there exists an $(\omega,\cc{P})$-admissable $\ra{S}_k$-path $(P,\alpha)$ with $V(P) = \{t_1,\ldots,t_n)$, $\alpha(t_i,t_{i+1}) = (A_i,B_i)$ for each $i$ and an $\omega$-diblockage $\cc{O}$ of $\ra{S}_k$ extending $\cc{P}$.

Since $\cc{O}$ extends $\cc{P}$, and $(P,\alpha)$ is $(\omega,\cc{P})$-admissable, it follows that $(A_1,B_1) \in \cc{P}^+ \cup \cc{P}_\omega^+ \subseteq \cc{O}^+$ and $(A_{n-1},B_{n-1}) \in \cc{P}^- \cup \cc{P}_\omega^- \subseteq \cc{O}^-$. Let $j = \max \{t \, : (A_t,B_t) \in \cc{O}^+ \}$. By the previous statement, $1 \leq j < n-1$, and so, since $\cc{O}$ is an orientation of $\ra{S}_k$, $(A_{j+1},B_{j+1}) \in \cc{O}^-$. However, $|A_{j+1} \cap B_j| < \omega$, contradicting the assumption that $\cc{O}$ is an $\omega$-diblockage of $\ra{S}_k$.

We will prove the statement by induction on $|\ra{S}_{\cc{P}}|$. We first note that we may assume that $\cc{P}_\omega^+ \subseteq \cc{P}^+$ and $\cc{P}_\omega^- \subseteq \cc{P}^-$. Indeed, if $\cc{P}_\omega^+ \not\subseteq \cc{P}^+$ then there is some separation $(A,B)$ with $|A| < \omega$ and $(A,B) \in \cc{P}^-$. However, we can then form an $(\omega,\cc{P})$-admissable $\ra{S}_k$-path as follows: Let $P=\{v_1,v_2\}$ and let $\alpha(v_1,v_2) = (A,B)$. It is a simple check that $(P,\alpha)$ is an $(\omega,\cc{P})$-admissable $\ra{S}_k$-path. A similar argument holds if $\cc{P}_\omega^- \not\subseteq \cc{P}^-$.

Suppose that $|\ra{S}_{\cc{P}}|= 0$, in which case $\cc{P}$ is a consistent orientation of $\ra{S}_k$. Suppose that $\cc{P}$ is not an $\omega$-diblockage. Since without loss of generality $\cc{P}_\omega^+ \subseteq \cc{P}^+$ and $\cc{P}_\omega^- \subseteq \cc{P}^-$ there must exist a pair $(A,B) \leq (C,D)$ with $(A,B) \in \cc{P}^+$, $(C,D) \in \cc{P}^-$ and $|B \cap C| < \omega$. Then, we can then form an $(\omega,\cc{P})$-admissable $\ra{S}_k$-path as follows: Let $P=\{v_1,v_2,v_3\}$ and let $\alpha(v_1,v_2) = (A,B)$ and $\alpha(v_2,v_3) = (C,D)$. It is a simple check that $(P,\alpha)$ is a $\cc{P}$-admissable $\ra{S}_k$-tree.

So, let us suppose that $|\ra{S}_{\cc{P}}| > 0$, and that there is no $\omega$-diblockage $\cc{O}$ of $\ra{S}_k$ extending $\cc{P}$. There exists some separation $(A,B) \in \ra{S}_k \setminus  (\cc{P}^+ \cup \cc{P}^-)$. Let us choose $(C,D) \leq (A,B)$ minimal with $(C,D) \in \ra{S}_k \setminus  (\cc{P}^+ \cup \cc{P}^-)$ and $(A,B) \leq (E,F)$ maximal with $(E,F) \in \ra{S}_k \setminus  (\cc{P}^+ \cup \cc{P}^-)$.

We claim that $\cc{P}_1 = (\cc{P}^+ \cup (C,D), \cc{P}^-)$ and $\cc{P}_2 = (\cc{P}^+, \cc{P}^- \cup (E,F))$ are both consistent partial orientations of $\ra{S}_k$. Indeed, by minimality of $(C,D)$ every separation $(U,V) < (C,D)$ is in $\cc{P}^+ \cup \cc{P}^-$ and since $(U,V) < (C,D) \not\in \cc{P}^+ \cup \cc{P}^-$ by the consistency of $\cc{P}$ it follows that $(U,V) \in \cc{P}^+$. Similarly $(U,V) \in \cc{P}^-$ for all $(E,F) < (U,V)$. Hence both $\cc{P}_1$ and $\cc{P}_2$ are consistent partial orientations of $\ra{S}_k$. Furthermore $|\ra{S}_{\cc{P}_1}|,|\ra{S}_{\cc{P}_2}| < |\ra{S}_{\cc{P}}|$. Therefore, we can apply the induction hypothesis to both $\cc{P}_1$ and $\cc{P}_2$.

Since an $\omega$-diblockage of $\ra{S}_k$ extending $\cc{P}_1$ or $\cc{P}_2$ also extends $\cc{P}$, we may assume that there exists an $(\omega,\cc{P}_1)$-admissable $\ra{S}_k$-path $(P_1,\alpha_1)$ and an $(\omega,\cc{P}_2)$-admissable  $\ra{S}_k$-path $(P_2,\alpha_2)$. Furthermore, we may assume that they are not $(\omega,\cc{P})$-admissable, and so the initial leaf separation of $(P_1,\alpha_1)$ is $(C,D)$ and the terminal leaf separation of $(P_2,\alpha_2)$ is $(E,F)$. Let us pick a separation $(C,D) \leq (X,Y) \leq (E,F)$ such that $|X,Y| = \lambda\big( (C,D),(E,F)\big)$. Note that, $(X,Y)$ is up-linked to $(C,D)$ and down-linked to $(E,F)$.

Let $(P'_1,\alpha'_1)$ be the up-shift of $(P_1,\alpha_1)$ onto $(X,Y)$ and let $(P'_2,\alpha'_2)$ be the down-shift of $(P_2,\alpha_2)$ onto $(X,Y)$. Note that the initial leaf separation of $(P'_1,\alpha'_1)$ and the terminal leaf separation of $(P'_2,\alpha'_2)$ are both $(X,Y)$. We form $(\hat{P}, \hat{\alpha})$ by taking  $\hat{P}$ to be the path formed by identifying the terminal leaf of $P'_2$ with the initial leaf of $P'_1$, with $\hat{\alpha}$ defined to be $\alpha'_1$ on $P'_1$ and $\alpha'_2$ on $P'_2$. We claim that $(\hat{P}, \hat{\alpha})$ is a $\cc{P}$-admissable $\ra{S}_k$ path. Let us write $V(\hat{P}) = \{\hat{t}_1,\hat{t}_2, \ldots, \hat{t}_{\hat{n}}\}$, with $\hat{\alpha}(\hat{t}_j,\hat{t}_{j+1}) = (\hat{A}_j,\hat{B}_j)$ for each $j$ and $\hat{B}_0 = \hat{A}_{\hat{n}} =  V(D)$.

For each $j \neq 1,\hat{n}$, the bag $\hat{A}_j \cap \hat{B}_{j-1}$ is the shift of some non-leaf bag in $(P_1,\alpha_1)$ or $(P_2,\alpha_2)$. Since these were $(\omega,\cc{P}_1)$ and $(\omega,\cc{P}_2)$-admissable respectively, the size of the bag was less than $\omega$. Therefore, since $(X,Y)$ is up-linked to $(C,D)$ and down-linked to $(E,F)$, by Lemma \ref{l:bags} $|\hat{A}_j \cap \hat{B}_{j-1}| \leq \omega$.

Finally, consider the separations $(\hat{A}_1,\hat{B}_1)$ and $(\hat{A}_{\hat{n}-1},\hat{B}_{\hat{n}-1})$. If we denote by $(U,V)$ the initial leaf separation in $(P_2,\alpha_2)$ then $(U,V) \in \cc{P}^+ \cup \cc{P}_0^+$, since $(P_2,\alpha_2)$ is $(\omega,\cc{P}_2)$-admissable. Then, $(\hat{A}_1,\hat{B}_1) = (U,V) \wedge (X,Y) \leq (U,V)$. We note that, since $\cc{P}$ is consistent, $\cc{P}^+$ is down-closed, as is $\cc{P}_0^+$ by inspection, and so it follows that  $(\hat{A}_1,\hat{B}_1) \in \cc{P}^+ \cup \cc{P}_0^+$. A similar argument shows that $(\hat{A}_{\hat{n}-1},\hat{B}_{\hat{n}-1}) \in \cc{P}^- \cup \cc{P}_0^-$.

\end{proof}

\section{Finding an arborescence as a butterfly minor}\label{s:arb}
In directed graphs, it is not clear what the best way to generalise the minor operation from undirected graphs. One suggestion (see for example Johnson, Robertson and Seymour \cite{JRST01}), is that of \emph{butterfly minors}. We say an edge $e = (u,v)$ in a digraph $D$ is \emph{contractible} if either $d^-(v) = 1$ or $d^+(u) = 1$ where $d^-$ and $d^+$ are the in- and out-degree respectively. We say a digraph $D'$ is a butterfly minor of $D$, which we write $D' \preceq D$, if $D'$ can be obtained from $D$ by a sequence of vertex deletions, edge deletions and contractions of contractible edges.

We say an $\ra{S}$-path $(P,\alpha)$ with $V(P) = \{t_1,t_2,\ldots, t_n\}$ and $\alpha(t_i,t_{i+1}) = (A_i,B_i)$ is a \emph{partial $\ra{S}$-path of width $< k$} if for each $ 2 \leq i \leq n-1$, $\omega_i := |A_i \cap B_{i-1}| \leq k$ and $\omega_n = |B_{n-1}| \leq k$. That is, it is a path-decomposition in which each bag except the first has size $\leq k$. Note that a partial $\ra{S}$-path of width $< k$ is necessarily an $\ra{S}_{k+1}$-path.

Let $\ra{S}'_{k+1} \subset \ra{S}$ be the set of $(A,B)$ such that $(A,B) = \alpha(t_1,t_2)$ for some partial $\ra{S}$-path of width $< k$. Note that, since the first bag in a partial $\ra{S}$-path of width $< k$ has size at most $k$, it follows that $|A,B| \leq k$ and hence $\ra{S}'_{k+1} \subset \ra{S}_{k+1}$.

\arb*
\begin{proof}
We may assume without loss of generality that $F$ is in fact an arborescence. Therefore, every vertex in $F$ apart from the root $v_0$ has exactly one in-neighbour, and there is some ordering of the vertices $V(F) = v_0,v_2,\ldots,v_n$ such that every $v_i$ has no in-neighbours in $\{v_{i+1}, \ldots, v_n \}$. Furthermore, without loss of generality we may assume that the digraph $D$ is weakly connected.

Let us define $(C_0,D_0)$ to be a $\leq$-minimal separation such that:
\begin{itemize}
\item $|C_0,D_0| = 0$;
\item $(C_0,D_0) \in \ra{S}'_{n+1}$.
\end{itemize}
Note that, since $(P,\alpha)$ with $V(P)=\{t_1,t_2\}$ and $\alpha(t_1,t_2) = (V,\emptyset)$ is a partial $\ra{S}$-path of width $< n$, at least once such separation exists. Let $x^0_0 \in C_0 \setminus D_0$, which is non-empty as dpw$(D) \geq n$.

We shall construct inductively $x^i_i$ and $(C_{i},D_{i})$ for $1 \leq i \leq n$ where $(C_i,D_i)$ is a $\leq$-minimal separation satisfying the following properties:
\begin{itemize}
\item $(C_i,D_i) \leq (C_{i-1},D_{i-1} \cup \{ x^{i-1}_{i-1}\}) \leq (C_{i-1},D_{i-1})$;
\item $|C_i,D_i| = i$;
\item $(C_i,D_i) \in \ra{S}'_{n+1}$.
\end{itemize}
Furthermore, we can label $C_i \cap D_i = \{ x^i_0, \ldots, x^i_{i-1} \}$ such that there exist vertex disjoint $x^{i-1}_j - x^i_j$ paths $P^i_j$ for each $j \leq i-1$. Finally, there is some $x^i_i \in C_i \setminus D_i$ such that if $v_i \in F$ has an in-neighbour $v_{k(i)}$ then there is an edge $(x^i_{k(i)},x^i_i) \in E(D)$.

Suppose we have constructed  $x^{i-1}_{i-1}$ and $(C_{i-1},D_{i-1})$. Since $(C_{i-1},D_{i-1}) \in \ra{S}'_{n+1}$ there is a partial $\ra{S}$-path $(P,\alpha)$ of width $< n$ such that $\alpha(t_1,t_2) = (C_{i-1},D_{i-1})$. Let $V(P) = \{ t_1, \ldots, t_m\}$, then we can form a partial $\ra{S}$-path $(P',\alpha')$ by letting $P' = \{t_0,t_1, \ldots, t_m\}$, $\alpha'(t_0,t_1) = (C_{i-1},D_{i-1} \cup \{ x^{i-1}_{i-1}\})$ and $\alpha' = \alpha$ on $P$.

Since $i < n$, it is clear that this is a partial $\ra{S}$-path of width $<n$, and so $(C_{i-1},D_{i-1} \cup \{ x^{i-1}_{i-1}\}) \in \ra{S}'_{n+1}$. Therefore, the set of separations satisfying the properties is non-empty, and so there is some $\leq$-minimal element $(C_i,D_i)$ which satisfies the three properties.

We claim that $\lambda((C_i,D_i),(C_{i-1},D_{i-1} \cup \{ x^{i-1}_{i-1}\})) = i$. Indeed, suppose for contradiction there exists $(C_i,D_i) < (X,Y) < (C_{i-1},D_{i-1} \cup \{ x^{i-1}_{i-1}\})$ with $|X,Y| = \lambda((C_i,D_i),(C_{i-1},D_{i-1} \cup \{ x^{i-1}_{i-1}\})) < i$. Note that, since $|C_i,D_i| = |C_{i-1},D_{i-1} \cup \{ x^{i-1}_{i-1}\}| = i$, the inequalities are strict.

By assumption, there is a partial $\ra{S}$-path $(P,\alpha)$ with of width $< n$ with initial leaf separation $(C_i,D_i)$. By construction $(X,Y)$ is up-linked to $(C_i,D_i)$, and so by Lemma \ref{l:bags} the up-shift of $(P,\alpha)$ onto $(X,Y)$ with respect to $(C_i,D_i)$ is a partial $\ra{S}$-path of width $< n$ with initial leaf separation $(X,Y)$. Hence $(X,Y) \in \ra{S}'_{n+1}$.

However, $|X,Y|:= \ell < i$ and $(X,Y) < (C_{i-1},D_{i-1}) \leq (C_\ell,D_\ell)$, contradicting the minimality of $(C_\ell,D_\ell)$. Therefore $\lambda((C_i,D_i),(C_{i-1},D_{i-1} \cup \{ x^{i-1}_{i-1}\})) = i$, and so by Menger's theorem there exists a family of vertex disjoint $C_i \cap D_i$ to $C_{i-1} \cap (D_{i-1} \cup \{ x^{i-1}_{i-1}\})$ paths. Let us label the vertices of $C_i \cap D_i =  \{ x^i_0, \ldots, x^i_{i-1} \}$ such that these paths are from $x^{i-1}_j$ to $x^i_j$. 

We claim that every $x^i_j$ with $j \leq i-1$ has an out-neighbour in $C_i \setminus D_i$. Indeed, suppose $x^i_j$ does not, then $(C_i \setminus \{x^i_j\},D_i) \in \ra{S}$. There exists a partial $\ra{S}$-path $(P,\alpha)$ with $V(P) = \{ t_1, \ldots, t_m \}$ of width $< n$ with initial leaf separation $(C_i,D_i)$ and so if we consider $(P,\alpha')$ with $\alpha'(t_1,t_2) = (C_i \setminus \{x^i_j\},D_i)$ and $\alpha = \alpha'$ on $P[\{t_2,\ldots,t_m\}]$, we see that $(P,\alpha')$ is also a partial $\ra{S}$-path $(P,\alpha)$ of width $<n$, with initial leaf separation $(C_i \setminus \{x^i_j\},D_i)$.

Therefore, $(C_i,D_i) > (C_i \setminus \{x^i_j\},D_i) \in \ra{S}'_{n+1}$, contradicting the minimality of $(C_{i-1},D_{i-1})$. Hence, if $v_{k(i)}$ is the in-neighbour of $v_i$, then we can pick $x^i_i \in C_i \setminus D_i$ such that there is an edge $(x^i_{k(i)},x^i_i) \in E(D)$.

For each $0 \leq j \leq n-1$ let $P_j = \bigcup_{i=j}^n P^i_j$. then $P_j$ is an $x^j_j$ to $x^n_j$ path containing $x^i_j$ for all $j \leq i \leq n-1$. Consider the subgraph of $D$ given by
\[
D' = \bigcup_{j=0}^n P_j \cup \{ (x^i_{k(i)}, x^i_i \,: \, 1 \leq i \leq n \}
\]

Note that $D'$ is an arborescence, and so each vertex has at most $1$ in-neighbour. Hence every edge is contractible, and by contracting each $P_i$ we obtain $F$ as a butterfly minor.
\end{proof}

One of the strengths of Theorem \ref{t:min} is that there exist forests with unbounded path-width, and so the theorem gives a family graphs that must appear as a minor of a graph with sufficiently large path-width, and conversely cannot appear as a minor of a graph with small path-width. 

Unfortunately, there do not exist arborescences of unbounded directed path-width. Indeed, since the vertices of an arborescence can be linearly ordered such that no edge goes `backwards', every arborescence has directed path-width $0$, and moreso this observation is even true for all directed acyclic graphs. It would be interesting to know if these methods could be used to prove a theorem similar to Theorem \ref{t:arborescence} for a class of graphs whose directed path-width is unbounded.

\section{Linked directed path-decompositions}\label{s:linked}
A directed graph $D$ is \emph{simple} if it is loopless and there is at most one edge $(u,v)$ for every $u,v \in V(D)$. A simple directed graph is \emph{semi-complete} if for every pair $u,v  \in V(D)$ either $(u,v) \in E(D)$ or $(v,u) \in E(D)$. A semi-complete directed graph is a \emph{tournament} if exactly one of $(u,v)$ and $(v,u)$ is an edge. Seymour and Kim \cite{KS15} considered the following property of a directed path-decomposition $(P,\cc{V})$:
\begin{equation}\label{e:linked}
\begin{split}
\text{If  } |V_k| \geq t &\text{ for every } i \leq k \leq j \text{, then there exists a collection of } t \\
&\text{ vertex-disjoint directed paths from } V_j \text{ to } V_i. 
\end{split}
\end{equation}

\begin{theorem}[Seymour and Kim]\label{t:SK}
Let $D=(V,E)$ be a semi-complete directed graph. $D$ has a directed path-decomposition $(P,\cc{V})$ of width dpw$(D)$ satisfying (\ref{e:linked}).
\end{theorem}

Seymour and Kim called directed path-decompositions satisfying (\ref{e:linked}), as well as two other technical conditions, `linked'. However, when thinking about directed path-decompositions in terms of separations, perhaps a more natural concept to call `linked' is the following (See \cite{E17}). We say a directed path-decomposition $(P,\cc{V})$ with $P = \{t_1,\ldots, t_n\}$ and $\alpha(t_i,t_{i+1}) = (A_i,B_i)$ is \emph{linked} if for every $1\leq i < j \leq n-1$
\[
\min \{|A_k,B_k| \, : \, i \leq k \leq j \} = \lambda\big( (A_i,B_i),(A_j,B_j) \big).
\]

It is easy to check that, if $(P,\alpha)$ is a linked directed path-decomposition then we can form a directed path-decomposition satisfying (\ref{e:linked}) without increasing the width in the following way. Let $P' = \{t_1,s_1,t_2,s_2, \ldots, t_{n-1},s_{n-1},t_n \}$ be the path formed by subdividing each edge in $P$ and let $\alpha'(t_i,s_i) = \alpha'(s_i,t_{i+1}) = \alpha(t_i,t_{i+1})$. In terms of the bags of the path-decomposition, this corresponds to adding the adhesion sets of the original path-decomposition as bags.

Given a directed path decomposition $(P,\alpha)$ and $r \in \mathbb{N}$ let us write $P_r$ for the linear sub-forest of $P$ induced by the edges $e \in E(P)$ such that $|\alpha(e)| \geq r$. Let us denote by $e(P_r)$ and $c(P_r)$ the number of edges and components of $P_r$ respectively.

Again we will prove a slightly more general theorem about directed path-decompositions where we fix independently the size of the adhesion sets. Theorem \ref{t:mainlinked} will follow from the following theorem if we let $\omega=k$. We note that a stronger result, which would imply Theorem \ref{t:mainlinked}, is claimed in a preprint of Kintali \cite{K14}. However we were unable to verify the proof, and include a counterexample to his claim in the appendix.
\begin{theorem}
Let $D=(V,E)$ be a directed graph and let $k \leq \omega \in \mathbb{N}$ be such that there exists an $\ra{S}_k$-path of width $< \omega - 1$. There exists a linked $\ra{S}_k$-path of width $< \omega - 1$.
\end{theorem}

\begin{proof}
Let us define a partial order on the set of $\ra{S}_k$-paths of $D$ of width $< \omega - 1$ by letting $(P,\alpha) \leq (Q,\beta)$ if there is some $r$ such that,
\begin{itemize}
\item for all $r' > r$, $e(P_{r'}) = e(Q_{r'})$ and $c(P_{r'}) = c(Q_{r'})$;
\item either $e(P_r) < e(Q_r)$, or $e(P_r) = e(Q_r)$ and $c(P_r) > c(Q_r)$;
\end{itemize}

Since there exists at least one directed path-decomposition of $D$ of width $< \omega -1$, there is some minimal element in this partial order, $(P,\alpha)$ with $V(P) = \{t_1,\ldots, t_n\}$ and $\alpha(t_i,t_{i+1}) = (A_i,B_i)$. We claim that $(P,\alpha)$ is linked.

Suppose for contradiction that $(P,\alpha)$ is not linked. That is, there exists $1 \leq i < j \leq n-1$ such that 
\[
\lambda\big( (A_i,B_i), (A_j,B_j) \big) < \min \{|A_k,B_k| \, : \, i \leq k \leq j \}.
\]
We will construct another directed path-decomposition $(\hat{P},\hat{\alpha})$ of width $< \omega -1$ such that $(\hat{P},\hat{\alpha})<(P,\alpha)$.

Let us choose a separation $(A_i,B_i) \leq (X,Y) \leq (A_j,B_j)$ such that $|X,Y| = \lambda \big( (A_i,B_i), (A_j,B_j) \big)$. Note that $(X,Y)$ is up-linked to $(A_i,B_i)$ and down-linked to $(A_j,B_j)$

We form two new $\ra{S_k}$-paths $(P',\alpha')$ and $(P'',\alpha'')$ by taking the up-shift of $(P,\alpha)$ onto $(X,Y)$ with respect to $(A_i,B_i)$ and the down-shift of $(P,\alpha)$ onto $(X,Y)$ with respect to $(A_j,B_j)$. Let us denote by $(A'_i,B'_i), \ldots , (A'_{n-1},B'_{n-1})$ and $(A''_1,B''_1) ,\ldots, (A''_j,B''_j)$ for the images of $\alpha'$ and $\alpha''$. We note that the initial leaf separation of $(P',\alpha')$ and the terminal leaf separation of $(P'',\alpha'')$ are both $(X,Y)$.

We form a new $\ra{S}_k$-path $(\hat{P},\hat{\alpha})$ by letting $\hat{P}$ be the path formed by identifying the initial leaf of $(P',\alpha')$ with the terminal leaf of $(P'',\alpha'')$ and taking $\hat{\alpha}$ to be $\alpha'$ on $E(P')$ and $\alpha''$ on $E(P'')$.

By Lemma \ref{l:bags}, since $(P,\alpha)$ was of width $< \omega-1$, so is $(\hat{P},\hat{\alpha})$. We claim that $(\hat{P},\hat{\alpha}) < (P,\alpha)$. Given a vertex $t_k \in V(P)$ we will write $t'_k$ and $t''_k$ for the copy of $t_k$ in $P'$ or $P''$ respectively, and carry these labels over onto $\hat{P}$. Note that, not every vertex will appear in both $P'$ and $P''$.

\begin{claim}\label{c:one}
For every $k \in [n]$ if $|A_k,B_k| = |A'_k,B'_k|$ then $|A''_k,B''_k| \leq |X,Y|$. Similarly if $|A_k,B_k| = |A''_k,B''_k|$ then $|A'_k,B'_k| \leq |X,Y|$.

Here we are assuming for convenience that if $(A'_k,B'_k)$ or $(A''_k,B''_k)$ do not exist then their order is $0$.
\end{claim}
\begin{proof}[Proof of claim]
We may assume that $i \leq k \leq j$, since otherwise one of $(A'_k,B'_k)$ or $(A''_k,B''_k)$ has order $0$.

Since $(P',\alpha')$ was the up-shift of $(P,\alpha)$ onto $(X,Y)$ with respect to $(A_i,B_i)$, $(A'_k,B'_k) = (A_k,B_k) \vee (X,Y)= (A_k \cup X, B_k \cap Y)$. Similarly $(A''_k,B''_k) = (A_k,B_k) \wedge (X,Y)= (A_k \cap X, B_k \cup Y)$.

Hence, by (\ref{e:sub}),
\[
|A'_k,B'_k| + |A''_k,B''_k| = |A_k,B_k| + |X,Y|.
\]
\end{proof}

Let us write $e_k$ for the edge $(t_k,t_{k+1})\in E(P)$, and $e'_k$, $e''_k$ for the two copies of $e_k$ in $E(\hat{P})$ (when they exist).

\begin{claim}\label{c:two}
For every $r > |X,Y|$ and every $k \in [n-1]$ such that $|A_k,B_k| = r$ exactly one of $(A'_k,B'_k)$, $(A''_k,B''_k)$ has order $|A_k,B_k|$, and the other has order $\leq |X,Y|$. Furthermore, for each component $C$ of $T^r$, and $e_p,e_q \in C$ as above, then $|A'_p,B'_p| = |A_p,B_p|$ if and only if $|A'_q,B'_q| = |A_q,B_q|$, and similarly $|A''_p,B''_p| = |A_p,B_p|$ if and only if $|A''_q,B'_q| = |A_q,B_q|$.
\end{claim}
\begin{proof}[proof of Claim]
We will prove the claim by reverse induction on $r$, starting with $r$ being the order of the largest separation in $(P,\alpha)$. Note that, since $|A_i,B_i|,|A_j,B_j| >  |X,Y|$, it follows that $r > |X,Y|$.

Since $(P,\alpha)$ was minimal, $e(P_r) \leq e(\hat{P}_r)$. However, by Claim \ref{c:one}, for each $(A_k,B_k)$ with $|A_k,B_k| = r$, at most one of the two separations $(A'_k,B'_k)$ and $(A''_k,B''_k)$ have order $r$, and if it does, then the other has order $\leq |X,Y| < r$. Therefore it follows that $e(\hat{P}_r) \leq e(P_r)$, and so $e(P_r) = e(\hat{P}_r)$, and the first part of the claim follows.

By minimality of $(P,\alpha)$ again, it follows that $c(P_r) \geq c(\hat{P}_r)$. However, since the edge $e'_i = e''_j$ in $\hat{P}$ is mapped to the separation $(X,Y)$ by $\hat{\alpha}$, it follows from the first half of the claim that $c(\hat{P}_r) \geq c(P_r)$, and so $c(P_r)= c(\hat{P}_r)$, and the second part of the claim follows.

Suppose then that the claim holds for all $r' > r$. It follows that $e(P_{r'}) = e(\hat{P}_{r'})$ and $c(P_{r'})= c(\hat{P}_{r'})$ for all $r'>r$ and so, since $(P,\alpha)$ was minimal, $e(P_r) \leq e(\hat{P}_r)$.

However, since by Lemma \ref{l:shift1} the order of each separation does not increase when we shift an $\ra{S}_k$-tree, the only edges in $\hat{P}_r$ come from copies of edges in $P_{r'}$ with $r' \geq r$. If $r' > r$ then, by the induction hypothesis, these copies have order $r'$, or $\leq |X,Y|$. If $r'=r$ then, by Claim \ref{c:one} at most one of the two copies of the edge has order $r$, and if it does the other has order $\leq |X,Y|$. It follows that $e(\hat{P}_r) \leq e(P_r)$, and so $e(P_r) = e(\hat{P}_r)$, and the first part of the claim follows.

By minimality of $(P,\alpha)$ again, it follows that $c(P_r) \geq c(\hat{P}_r)$. However, since the edge $e'_i = e''_j$ in $\hat{P}_r$ is mapped to the separation $(X,Y)$ by $\hat{\alpha}$, it follows from the first half of the claim that $c(\hat{P}_r) \geq c(P_r)$, and so $c(P_r)= c(\hat{P}_r)$, and the second part of the claim follows.
\end{proof}

Recall that $|A_k,B_k| > |X,Y|$ for all $i\leq k \leq j$ by assumption. Hence $e_i$ and $e_j$ lie in the same component of $P_{|X,Y| +1}$. However, $(A'_i,B'_i) = (A''_j,B''_j) = (X,Y)$, contradicting the second part of Claim \ref{c:two}.
\end{proof}

\appendix
\section{Lean directed path-decompositions}
Kintali \cite{K14} defines a directed path-decomposition to be \emph{lean} if it satisfies the following condition:
\begin{equation}\label{e:lean}
\begin{split}
&\text{Given $k > 0 $, $t_1\leq t_2 \in [n]$ and subsets $Z_1 \subset C_{t_1}$ and $Z_2 \subset C_{t_2}$ with $|Z_1|= |Z_2| = k$} \\
&\text{either $G$ contains $k$ vertex-disjoint directed paths from $Z_2$ to $Z_1$ or there exists}\\
&\text{$i \in [t_1,t_2-1]$ such that $|V_i \cap V_{i+1}| < k$ }
\end{split}
\end{equation}
Note, this is a strengthening of (\ref{e:linked}). In particular, (\ref{e:lean}) has content in the case $t_1= t_2$. When Thomas proved his result on the existence of linked tree-decompositions of minimal width \cite{T90} he in fact established the existence of tree-decompositions satisfying a stronger condition in the vein of (\ref{e:lean}) (which are sometimes called \emph{lean} tree-decompositions in the literature \cite{D00}). Kintali claims the following analogous result.

\begin{theorem}[\cite{K14} Theorem 7]
Every digraph $D$ has a directed path-decomposition of width dpw$(D)$ satisfying (\ref{e:lean}).
\end{theorem}

However, we note that this theorem cannot hold. Indeed, consider a perfect binary tree of depth $n$, with all edges in both directions. Let us write $T_n$ for the undirected tree and $\ra{T}_n$ for the directed graph. It it easy to see that the directed path-width of $\ra{T}_n$ is equal to the path width of $T_n$, which is $\frac{n-1}{2}$. Hence, in any directed path-decomposition of $\ra{T}_n$ there is some bag of size at least $\frac{n+1}{2}$. Suppose that a lean directed path-decomposition exists, let us denote by $V_i$ the bag such that $|V_i| \geq \frac{n+1}{2}$.

If we consider $V_i$ as a subset of $T_n$, then it follows from (\ref{e:lean}) that for every $k>0$ and every $Z_1,Z_2 \subset V_t$ with $|Z_1|= |Z_2| = k$, $T_n$ contains $k$ vertex-disjoint paths between $Z_1$ and $Z_2$. This property is known in the literature as being \emph{well-linked}, and the size of the largest well-linked set in a graph is linearly related to the tree-width (see for example \cite{HW17}). Specifically, since $T_n$ contains a well-linked set of size $\geq \frac{n+1}{2}$ it follows that tw$(T_n) \geq \frac{n+1}{6}$. However, the tree-width of any tree is one, contradicting the existence of a lean directed path-decomposition of $\ra{T}_n$ for $n \geq 6$.

\bibliographystyle{plain}
\bibliography{dual}

\end{document}